\documentclass[12pt]{amsart}
\UseRawInputEncoding
\usepackage[usenames]{color}
\input xypic

\usepackage{verbatim}
\usepackage{url}
\usepackage{bm}
\usepackage{tikz-cd}
\usepackage{lipsum}

\usepackage{amsmath}
\usepackage[centertags]{amsmath}
\usepackage{amsfonts}
\usepackage{amssymb}
\usepackage{amsthm}
\usepackage{newlfont}
\usepackage{mathtools}
\usepackage{url}

\theoremstyle{plain}
\newtheorem{thm}{Theorem}
\newtheorem{cor}[thm]{Corollary}
\newtheorem{conj}[thm]{Conjecture}
\newtheorem{lem}[thm]{Lemma}
\newtheorem{lem*}[thm]{Lemma}

\newtheorem{ass}[thm]{Assumptions}
\newtheorem{assumption}[thm]{Assumptions}

\theoremstyle{definition}
\newtheorem{dfn}{Definition}
\theoremstyle{remark}
\newtheorem{rem}{Remark}
\newtheorem{rem*}{Remark}
\newtheorem{ex}[rem]{Example}

\numberwithin{rem}{section} 
\numberwithin{dfn}{section} 
\numberwithin{equation}{section} 
\numberwithin{thm}{section} 

\def\!{\operatorname{!}}

\def\Z{\mathbb Z}

\def\1{\bold 1}

\def\ord{\operatorname{ord}}

\def\Hom{\operatorname{Hom}}

\def\Im{\operatorname{Im}}

\def\tr{\operatorname{tr}}

\def\End{\operatorname{End}}

\def\res{\operatorname{res}}

\def\mod{\operatorname{mod}}

\begin{document}

\title{Note on  linear relations in Galois cohomology and {\'e}tale $K$-theory of curves}

\author{Piotr Kraso\'n}
\begin{abstract}
In this paper we investigate a local to global principle for Galois cohomology of number fields with coefficients in the Tate module of an abelian variety. In \cite{bk13}  G. Banaszak and the  author obtained the sufficient condition for the validity of  the local to global principle for {\'e}tale $K$-theory of a curve . This condition in fact  has been established by means of
an analysis of the corresponding problem in the Galois cohomology. We show that in some cases this result is the best possible i.e 
if this condition does not hold we obtain  counterexamples.
 We also give some examples of curves and their Jacobians.  Finally, we prove the dynamical version of the local to global principle for {\'e}tale $K$-theory of a curve. The dynamical local to global principle for the groups of Mordell-Weil type has recently been considered by S. Bara{\'n}czuk in \cite{b17}. We show that all our results remain valid for Quillen $K$-theory of ${\cal X}$ if the Bass and Quillen-Lichtenbaum conjectures hold true for ${\cal X}.$
\end{abstract}
\date{\today}

\address{Institute of Mathematics, University of Szczecin, ul. Wielkopolska 15, 70-451 Szczecin, Poland }
\email{ piotrkras26@gmail.com}


\maketitle

\section{Introduction} 

The local to global type questions are of mathematical interest since the celebrated Hasse principle was proven. For the history of these type of  problems in number theory and its extensions 
to the context of abelian varieties and linear algebraic groups see \cite{fk17}.

In \cite{bk13} G. Banaszak and the  author proved a sufficient condition for the local to global condition to hold for {\'e}tale $K$-theory of curves. This was done via analysis of the reduction map in the suitable Galois cohomology.
 The main result of the current   paper is Theorem \ref{best}. It asserts, for the local to global questions considered there, that if the sufficient  condition  is not fulfilled then one can produce   a counterexample. This gives some evidence that the criterion obtained in \cite{bk13} is the best possible  (cf. Remark \ref{best1}).

To state the  aforementioned local to global principle for {\'e}tale $K$-theory of a curve let us recall the following definitions and notations from \cite{bk13}.
Let $X/F$ be a smooth, proper and geometrically irreducible curve of genus $g$ defined over a number field $F$ and let $J$ be the Jacobian  of $X.$

\begin{dfn}
We call a finite field extension $F^{\prime} / F$ an {\bf isogeny splitting field} of the Jacobian $J$ if $J$ is isogenous 
over $F^{\prime}$ to the product  $A_{1}^{e_{1}}\times \dots \times A_{t}^{e_{t}}$ where $A_1, \dots, A_t$ are pairwise nonisogenous, 
absolutely simple abelian varieties defined over $F^{\prime}$.
\label{definition of an isogeny splitting field}\end{dfn}

\begin{rem} Notice that $J$ is an abelian variety and therefore by Poincar{\'e} decomposition theorem the isogeny  splitting field exists.
\end{rem}
Let  ${\tilde V}_{l,\, i}$ denote $T_{l}(A_{i})(n)\otimes_{\Z_l} {\mathbb Q}_{l},$ where  $T_{l}(A_{i})(n)$ is the $n$-th twist of the $l$-adic Tate module of $A_{i}$ (cf. \cite{t68}).

\medskip

The main result of \cite{bk13} is the following theorem: 
\begin{thm}\label{theorem 1.1}
Let $X/F$ be a smooth, proper and geometrically irreducible curve of genus $g.$ Let $F^{\prime}$ be an isogeny splitting field of the Jacobian $J$ and 
assume that for the corresponding product $A_{1}^{e_{1}}\times \dots \times A_{t}^{e_{t}},$ we have $End_{{\overline F}} \, A_{i}={\mathbb Z}$ for each $1 \leq i \leq t.$ Let $l > 2$ be a prime number which is coprime to the polarisation degrees of the abelian varieties 
$A_i.$ Let $S_{l}$ be a set of places of $F$ containing the places of bad reduction, the  
archimedean places and the primes above $l.$ Let  ${\cal X}$ be a regular and proper model of $X$ over ${\cal O}_{F, S_l}.$ 
Assume that ${{dim}}_{{\mathbb Q}_{l}}{\tilde V}_{l,\, i} \, \, \geq e_{i}$ for each $1 \leq i \leq t.$ 
 Let $\hat{P} \in K_{2n}^{et}({\cal X})$ and let $\hat{\Lambda}$ be a finitely generated 
 ${\mathbb Z}_{l}$-submodule of $ K_{2n}^{et}({\cal X}).$ If $r_v (\hat{P}) \in
 r_v (\hat{\Lambda})$ for almost all $v$ of  ${\cal O}_{F,S_{l}}$ then $\hat{P} \in 
 \hat{\Lambda} + K_{2n}^{et}({\cal X}) _{tor}.$
\end{thm}
Theorem \ref{theorem 1.1} is equivalent with  the following (cf. diagram (\ref{diagram 2.4})  and \cite{bk13} especially section 4):
\begin{thm}\label{theorem 1.2}
Let $X/F,$  be a curve and $S_l$ the set of primes as in Theorem \ref{theorem 1.1}. Let $F^{\prime}$ be  an isogeny splitting field i.e. $J$ is isogenous over $F^{\prime}$ to 
$ A_1^{e_1}\times\dots\times A_t^{e_t}$  and  for all $1\leq i \leq t$ we have  $End_{{\overline F}} \, A_{i}={\mathbb Z}$ and
 ${{dim}}_{{\mathbb Q}_{l}}{\tilde V}_{l,\, i} \, \, \ge e_{i}.$ Let ${\hat P}\in H^{1}(G_{S_{l}}; {T}_{l}(J)(n))$ and  ${\hat\Lambda}$ be a finitely generated 
 ${\mathbb Z}_{l}$-submodule  of $H^{1}(G_{S_{l}}; {T}_{l}(J)(n)).$ If $r_v (\hat{P}) \in
 r_v ({\hat\Lambda})\subset H^{1}({g}_{v}; {T}_{l}(J_{v})(n )) $ for almost all $v$ of  ${\cal O}_{F,S_{l}}$ then $\hat{P} \in 
 \hat{\Lambda} + H^{1}(G_{S_{l}}; {T}_{l}(J)(n))_{tor}.$
\end{thm}
For the definition of $G_{S_{l}}$ and $g_v$ see section 3.

We prove the following theorem:
\begin{thm}\label{best}
Let $A=A_1^{e_1} \times\dots \times A_t^{e_{t}}$ be an abelian variety such that $End_{{\overline F}} \, A_{i}={\mathbb Z}$ for all $i.$  Let $S_l$ be the set of primes containing archimedean places, primes of bad reduction of $A$ and primes over $l.$ Assume that   ${\mathrm{rank}}A_i( F)\ge e_i$ and 
 ${{dim}}_{{\mathbb Q}_{l}}{\tilde V}_{l,\, i} \, \, < e_{i}$ for some $1\leq i \leq t.$    
 Then:
 \begin{enumerate}
 \item
 the local to global principle for the map $H^{1}(G_{S_{l}}; {T}_{l}(A))\rightarrow H^{1}({g}_{v}; {T}_{l}(A_{v}))$ does not hold i.e. there exists $\Lambda \subset H^{1}(G_{S_{l}}; {T}_{l}(A))$ and $P\in H^{1}(G_{S_{l}}; {T}_{l}(A))$ such that $r_v(P)
\in r_v(\Lambda) $ but $P\notin\Lambda+H^{1}(G_{S_{l}}; {T}_{l}(J)(n))_{tor}$  
  \item 
  the local to global principle for the map $H^{1}(G_{F_{{l^{\infty},S_{l}}}}; {T}_{l}(A)(n))\rightarrow H^{1}({g}_{v}; {T}_{l}(A_{v})(n)),$   where $F_{{l^{\infty},S_{l}}}:={\bigcup}_{n=1}^{\infty}F_{S_{l}}({\mu}_{l^{n}})$ and ${\mu}_{l^{n}}$ is the $l^n$-th root of unity, does not hold.
   \end{enumerate}
\end{thm}

\begin{rem}
Without loss  of generality in the  proof of  Theorem \ref{best} we may assume $F^{\prime}=F$.  From now on we assume this.
\end{rem}
\begin{rem}\label{best1}
The reason that we could not  produce a counterexample for an arbitrary Tate twist i.e. for the  map 
$r_v: H^{1}(G_{{S_{l}}}; {T}_{l}(A)(n))\rightarrow H^{1}({g}_{v}; {T}_{l}(A_{v})(n))$ is the difficulty 
in estimating the rank of $H^{1}(G_{S_{l}}; {T}_{l}(A)(n)).$ However, if this rank is sufficiently big then, similar to that given in section 6, proof shows that  the corresponding to (1) result holds true. Then abelian varieties from Example \ref{ex6} yield, if the numeric condition of Theorem \ref{theorem 1.1} is violated, counterexamples to the local to global principle for even {\'e}tale  $K$-theory of a curve.
\end{rem}

In section 2 we give some examples of Jacobians of curves satisfying the assumptions of Theorem \ref{theorem 1.1} and examples satisfying assumptions of Theorem \ref{best}. Sections 3, 4 and 5 contain necessary background from Galois cohomology, intermediate Jacobians and
Kummer theory. In section 6 we prove our main theorem (Theorem \ref{best}).  
Section 7 is devoted to proof of the dynamical version of the local to global principle. This is done by checking the axioms for the dynamical local to global principle introduced in \cite{b17}. We check these axioms for Galois 
cohomology (cf. Lemma \ref{kg})  and then pass to the {\'e}tale $K$-theory. We also show that we can replace {\'e}tale $K$-theory by Quillen theory of a curve provided Bass and Quillen-Lichtenbaum conjectures hold true.

\section{Examples}
In this section we give examples of curves whose Jacobian decomposes as $A_1^{e_1}\times \dots A_t^{e_t},$  with ${\End}_{\overline F}A_i={\mathbb Z}.$
In general deciding whether the Jacobian of a curve  splits is a difficult problem and has vast literature (cf. e.g. \cite{hn65}, \cite{es93}).
However, we have the following \cite{co12}, p.589:

\begin{thm}\label{isogeny}
Any abelian variety $A$ defined over $\overline{\mathbb Q}$ of dimension $1\leq g\leq 3$ is isogenous to $J(C)$, where $C$ is a curve of genus $g$.
\end{thm}
\begin{rem}
Genus $1$ case of Theorem \ref{isogeny} is trivial, genus $2$ is \cite{w57} Satz 2 and $g=3$ is Theorem 4  of \cite{ou73}.
\end{rem}

\begin{ex}\label{ex1}
Let $A/F$ be principally polarized abelian surface with ${\End_{{\overline F}}}A={\mathbb Z}$ then by Theorem \ref{isogeny} there exists a curve $X$ defined over $F^{'} \subset {\overline Q}$ such 
J(X) is isogenous to $A.$
\end{ex}
\begin{ex}\label{Ex2}
Let $B=A\times E$, where $A$ is as in Example \ref{ex1} and $E$ is an elliptic curve without complex multiplication then there exists a curve $X$ such that J(X) is isogenous to $B.$ 
\end{ex}
A special case of Example \ref{ex1} is given in the following (see \cite{bpp18} section 7.1 p. 39):
\begin{ex}\label{ex3}
Let $X/{\mathbb Q}$ be a smooth projective curve given by the following equation
\begin{equation}\label{ex4}
X:  \quad y^2 +(x^3+x^2 +x+1)y=-x^2-x.
\end{equation}
Then J(X)  is a principally polarized abelian isotypical  surface over ${\mathbb Q}$ of conductor 277.
\end{ex}
\begin{rem}\label{rmm} Since the ${\mathbf j}$ invariant of a curve with complex multiplication is an algebraic integer (cf.  \cite{c89}, Theorem 11.1 ), in Example \ref{Ex2} it is enough to pick an elliptic curve defined over ${\mathbb Q}$ whose ${\mathbf j}$ invariant is non-integral rational number.
\end{rem}
The following lemma        guarantees that there are  many simple abelian surfaces $A$ with ${\End}_{\overline F} A={\mathbb Z}$
\begin{lem}(\cite{bpp18}, Lemma 4.1.1)
Let $A$ be a simple, semistable abelian surface over ${\mathbb Q}$ with non-square conductor then $A$ is isotypical, i.e. ${\End}_{\overline{\mathbb Q}}A={\mathbb Z}$.
\end{lem}
Now we can give  examples of abelian varieties satisfying the assumptions of Theorem \ref{best}
\begin{ex}\label{ex6}
In Example \ref{Ex2} take $A$ to be any simple abelian surface $A$ over ${\mathbb Q}$ with ${\End}A={\mathbb Z}$ and $E/{\mathbb Q}$ to be any elliptic curve without complex multiplication  such that ${\mathrm{rank}}E(\mathbb Q)>2.$ 
One specific choice of that sort will be $A=J(X)$ where $X$ is a curve from Example \ref{ex3} and $E$ the elliptic curve 
(listed in Cremona tables) of conductor  5077  with the coefficients  $[a_1,a_2, a_3. a_4, a_6]=[0,0,1,-7,6]$ and the Mordell-Weil group over ${\mathbb Q}$ of rank 3. One readily verifies that the $j$-invariant is not a rational integer, so this curve has no CM (cf. Remark \ref{rmm}).

\end{ex}

\section{necessary results concerning {\'e}tale $K$-theory of curves and cohomology}

We start with the definition of continuous cohomology \cite{j88} and \cite{df85}. 
\begin{dfn}
Let $Y$ be a scheme over $\Z[\frac{1}{l}]$ and let $({\cal F}_m)$ be a projective system of 
${\Bbb Z}/{l^{m}}$ - {\' e}tale sheaves ${\cal F}_m$. The functor $({\cal F}_m) \rightarrow 
\varprojlim H^{0}(Y,\, ({\cal F}_m))$ is left exact and its $i$-th right derived functor is by definition the continuous cohomology group $H^{i}_{cts}(Y,\, ({\cal F}_m)).$\end{dfn}
\medskip

The {\' e}tale K-theory spaces and groups were defined in \cite{df85} by W. Dwyer and E. Friedlander.
In \cite{bgk99} using the spectral sequences of \cite{df85}:
$$E_{2}^{p,-q} = H^{p}_{cont} (Y, \, \Z_l ({q}\, {/}{2})) \Rightarrow K_{q-p}^{et} (Y),$$
$$E_{2}^{p,-q} = H^{p}_{et} (Y, \, \Z/l^k ({q}\, {/}{2})) \Rightarrow K_{q-p}^{et} (Y, \, \Z/l^k)\,,$$
it was shown that
one has the following  exact sequence connecting continuous cohomology of $\cal X$ and  the 
Galois cohomology of the Galois group of the maximal unramified outside $S_{l}$ extension of $F:$ 

\begin{equation}\label{diagram 2.1}
0\rightarrow H^{2}(G_{S_{l}} ; {\Bbb Z}_{l}(k)) \rightarrow H^{2}_{cts}({\cal X}; {\Bbb Z}_{l}(k)) \rightarrow H^{1}(G_{S_{l}} ; {T}_{l}(J)(k-1))\rightarrow 0 \, ,
\end{equation}
as well as existence of the following commutative diagram:

\begin{equation}\label{diagram 2.4}
\xymatrix{
K_{2n}^{et}({\cal X}) \ar[d] \ar[r]^{r_v} & K_{2n}^{et}({\cal X}_{v}) \ar[d]^{\cong} \\
H^{1}(G_{S_{l}}; {T}_{l}(J)(n)) \ar[r]^{r_v}&  H^{1}({g}_{v}; {T}_{l}(J_{v})(n ))\, ,
}
\end{equation} where $g_v:=Gal({\bar{k_v}}/k_v).$
The right hand vertical arrow in the diagram  (\ref{diagram 2.4}) is an isomorphism whereas the left vertical arrow is an epimorphism with finite kernel (cf. \cite{bk13}, \cite{bgk99}, \cite{bg08}).

One has the following Dwyer-Friedlander homomorphisms, for odd prime $l$ and $k\geq 1$, connecting Quillen $K$-theory  with the {\'e}tale $K$-theory with coefficients (cf. \cite{df85} ):
\begin{equation}\label{dfr}
h_{l^k} : K_m({\cal X}, {\mathbb Z}/{l^k{\mathbb Z}}) \rightarrow  K_m^{et}({\cal X}, {\mathbb Z}/{l^k{\mathbb Z}}), 
\end{equation}
\begin{equation}\label{dfr0}
{\overline{h_{l^k}}} : K_m({\cal X}_{v}, {\mathbb Z}/{l^k{\mathbb Z}}) \rightarrow  K_m^{et}({\cal X}_{v}, {\mathbb Z}/{l^k{\mathbb Z}}).
\end{equation}

In the sequel we assume the following: 
\begin{conj}\label{Bass}(Bass conjecture)
The Quillen $K$-theory groups $K_m({\cal X})$ are finitely generated for $m>0$.
\end{conj}
Assuming Conjecture {\ref{Bass} we obtain the following equality:
\begin{equation}
\varprojlim K_m({\cal X},{\mathbb Z}/{l^k{\mathbb Z}})=K_m({\cal X})\otimes {\mathbb Z}_l.
\end{equation}

The Dwyer-Friedlander homorphisms (\ref{dfr}) and (\ref{dfr0}) induce the following homomorphisms:
\begin{equation}\label{dfr1}
h : K_{2n}({\cal X})\otimes {\mathbb Z}_l \rightarrow  K_{2n}^{et}({\cal X}) 
\end{equation}
\begin{equation}\label{dfr2}
{\overline{h}} : K_{2n}({\cal X}_v)\otimes {\mathbb Z}_l\rightarrow  K_{2n}^{et}({\cal X}_{v})
\end{equation}}
The maps (\ref{dfr1}) and (\ref{dfr2}) are isomorphism if we assume that the Quillen-Lichtenbaum conjecture holds true for $\cal X.$
Thus we obtain the following commutative diagram:
\begin{equation}\label{diagram 2.5a}
\xymatrix{
K_{2n}({\cal X}) \ar[d]_{h^{'} }\ar[r]^{r_v} & K_{2n}({\cal X}_{v}) \ar[d]^{\overline{h^{'}} } \\
 K_{2n}^{et}({\cal X})\ar[r]^{r_v} &  \,  K_{2n}^{et}({\cal X}_v)
}
\end{equation}
where the map $h^{'}$ (resp. ${\overline{h^{'}} }$ ) is the composition of the natural homomorphism 
$K_{2n}({\cal X})\rightarrow K_{2n}({\cal X})\otimes {\mathbb Z}_l$ (resp. $K_{2n}({\cal X}_v)\rightarrow K_{2n}({\cal X}_v)\otimes {\mathbb Z}_l$) with $h$ (resp. $\overline{h}$).

Concatenation of diagrams (\ref{diagram 2.4}) and (\ref{diagram 2.5a})  yields the following commutative diagram:
\begin{equation}\label{diagram 2.4a}
\xymatrix{
K_{2n}({\cal X}) \ar[d] \ar[r]^{r_v} & K_{2n}({\cal X}_{v}) \ar[d]\\
H^{1}(G_{S_{l}}; {T}_{l}(J)(n)) \ar[r]^{r_v} &  H^{1}({g}_{v}; {T}_{l}(J_{v})(n ))\, 
}
\end{equation}
with the vertical maps having finite kernels.

\section{Intermediate Jacobian}
In this section to ease notation we denote ${\tilde T}_{l,i}:= T_l(A_i)(n).$
Let $L$ be a finite extension of $F.$
For $w\notin S_{l} ,$ let $G_{w} := 
G({\overline L_{w}}/L_w)$ be the absolute Galois group of the completion $L_{w}$ of $L$ at $w.$ Let $k_{w}$ be the residue field at $w.$ Put $H^{1}_{f}(G_{w};{\tilde T}_{l,\, i})=i_{w}^{-1}H^{1}_{f}(G_w ;{\tilde V}_{l,\, i}).$ Here $i_{w}: H^{1}(G_{w};{\tilde T}_{l,\, i})\rightarrow H^{1}(G_{w};{\tilde V}_{l,\, i})$ and $H^{1}_{f}(G_{w};{\tilde V}_{l,\, i})=ker(\, res: H^{1}(G_{w}; {\tilde V}_{l,\, i})\rightarrow 
H^{1}(I_{w}; {\tilde V}_{l,\, i}) )$ where $I_{w} \subset G_{w}$ is the inertia subgroup and $res$
is the restriction map. Let
\setcounter{equation}{0}
\begin{equation}\label{equation 3.2}
H^{1}_{f,S_{l}}(G_{L}; {\tilde T}_{l,\, i}) = ker(\, H^{1}(G_{L}; {\tilde T}_{l,\, i})\rightarrow \prod_{w\notin S_{l}} H^1(G_{w}; {\tilde T}_{l,\, i})/H^1_{f}(G_{w}; {\tilde T}_{l,\, i}) \, )\, .
\end{equation} 
We define the intermediate Jacobian (see \cite{bgk05}, \cite{bgk03}):
\begin{equation}\label{equation 3.3}
J_{f,S_{l}} ({\tilde T}_{l,\, i}) = \lim_ {{\longrightarrow}\atop {L/F}}H^{1}_{f,S_{l}}(G_{L}; {\tilde T}_{l,\, i})\, .
\end{equation} 
In \cite{bgk05}, \cite{bgk03}, in  the more general situation of any free ${\mathbb Z}_{l}$-module of finite rank  $T_{l},$ we made the following:
\begin{assumption}\label{assumption 3.2}
For any finite extension $L/F$ and any place $w\in {\mathcal O}_L$ such that $w\notin S_{l}\, ,$  we have $T_{l}^{Fr_{w}}=0.$
\end{assumption}
\begin{rem}\label{remm}
We have  $H^{1}_{f,S_{l}}(G_{L}; {\tilde T}_{l,\, i}) \cong H^{1}(G_{L,S_{l}}; {\tilde T}_{l,\, i})$  and the assumption \ref{assumption 3.2} is satisfied 
(cf. \cite{bgk05}, p.5 ).
\end{rem}
 By \cite[Prop. 2.14]{bgk03} we have the following isomorphisms:
\begin{equation}
J_{f,S_{l}}({\tilde T}_{l,\, i})_{l} \,\, \cong \,\, 
{\tilde V}_{l,\, i} / {\tilde T}_{l,\, i} \,\, \cong \,\, 
V_{l}(A_i) / T_{l}(A_i) \, (n) \,\, \cong \,\, A_i [l^{\infty}] (n) \, .
\label{torsion of interm Jacobian}
\end{equation} 

Let ${\tilde T}_{l} := \bigoplus_{i=1}^t \, {\tilde T}_{l,\, i}.$ Since each ${\tilde T}_{l,\, i}$  is a free ${\mathbb Z}_l (n)$-module we have the following  natural isomorphism:
\begin{equation}
J_{f,S_{l}} ({\tilde T}_{l}) \cong \bigoplus_{i=1}^t \, J_{f,S_{l}} ({\tilde T}_{l,\, i})\, .
\label{intermediate jacobian decomposition}\end{equation}   

For a finite extension $L/F$ the  following map corresponding to these used in the direct systems
 (\ref{equation 3.3}): 
\begin{equation}\label{i*}
i^{*}: H^{1}(G_{F, S_{l}}, {T}_{l}(J)(n)) \rightarrow H^{1}(G_{L, S_{l}}, {T}_{l}(J)(n))
\end{equation}
is an injection. 
One readily verifies, using transfer in the Galois cohomology, that the map (\ref{i*}) has $l$-torsion kernel.. 
But the $l$-torsion part of (\ref{i*}) is just 
\begin{equation}\label{lpart}
H^{0}(G_{F, S_{l}}, {V}_{l}(J)/{T}_{l}(J) \, (n)) \rightarrow H^{0}(G_{L, S_{l}}, {V}_{l}(J) / {T}_{l}(J) \, (n)),
\end{equation}
which is clearly injective.
Similarly, the reduction of the map (\ref{i*}) 

 \begin{equation}\label{inject}
H^{1}({g}_{v}, {{T}_{l}(J_{v})}(n )) \rightarrow H^{1}({g}_{w}, {T}_{l}(J_{v})(n ))
 \end{equation}
is  injective for any finite extension $L/F$ and any prime $w$ of $\mathcal{O}_L$ over $v,\, v\notin S_l.$
Notice that by Definition \ref{definition of an isogeny splitting field} we have
\begin{equation}
H^{1}(G_{F, S_{l}}; {T}_{l}(J)(n)) \cong 
H^{1}(G_{F, S_{l}}; {T}_{l}(A) (n)) = \prod_{i=1}^{t} 
H^{1}(G_{F, S_{l}}; {T}_{l}(A_i)(n))^{e_i}.
\label{The Zl module}\end{equation}
Let us also recall   the following:
\begin{lem}\label{inj}( \cite{bgk03}, Lemma 2.13)
For any finite extension $L/F$ and  any prime
$w \notin S_l$ in ${\cal O}_{L}$ the natural map
\begin{equation}
r_{w}\,\,:\,\,H^1_{f, S_l}(G_L; T_l)_{tor}\rightarrow H^1(g_w;\, T_l)
\end{equation}
is an imbedding.
\end{lem}

\section{Kummer theory}
Kummer theory in the context of abelian varieties was developed in \cite{r79}. 
In this section we collect necessary facts which will be useful in section 7.

\begin{dfn}\label{kummer}
 Let $A$ be an abelian variety over a number field $F$. Let ${\cal R}=\End_FA$. For $\alpha\in {\cal R}$ we set $F_{\alpha}=F(A[\alpha]) $ and $ G_{\alpha} = Gal(F_{\alpha}/F)$. The Kummer map:
\begin{equation}\label{psi}
{\psi}^{\alpha} : A(F )/{\alpha}A(F)\rightarrow {\Hom}_{G_{\alpha}} (􏰂G_{F_{\alpha}}, A[\alpha])
\end{equation}
 is defined as the composition:
 \begin{equation}
A(F)/{\alpha}A(F)\hookrightarrow  􏰇H^1( F,A[\alpha]) \xrightarrow{\res} H^1( F_{\alpha},A[\alpha])^{G_{\alpha}}
\end{equation}
with the first map a coboundary map for the $G_F$-cohomology of the Kummer sequence:
\begin{equation}\label{ks}
0\rightarrow A [ \alpha ] \rightarrow A ({\overline F} ) \xrightarrow{\alpha} A ({\overline F}   ) \rightarrow 0
\end{equation}
and the second map the restriction to $F_{\alpha}$.
\end{dfn}
Explicitly, the map (\ref{psi}) is given by the following formula
\begin{equation}\label{explicit}
{\psi}^{\alpha}({x} )(\sigma)={\sigma}(\frac{x}{\alpha})-\frac{x}{\alpha},
\end{equation}
where $\frac{x}{\alpha}$ is a fixed "$\alpha$-root" of $x$ i.e. an element $y\in A({\overline F})$ such that ${\alpha}y=x.$

We are interested in Kummer maps where $l$ is a rational prime and  ${\alpha}=l^k\in {\mathbb Z}\hookrightarrow  {\cal R}.$
Thus we have the  the family of Kummer maps:
\begin{equation}\label{ladic}
{\psi}^{l^k} :  A(F )/{l^k}A(F)\rightarrow {\Hom}_{G_{l^k}} (􏰂G_{F_{l^k}}, A[l^k]).
\end{equation}
The maps (\ref{ladic}) are compatible with the natural maps induced by multiplication by $l$. Therefore taking the inverse limit of both sides and twisting it with ${\mathbb Z}_l(n)$ yields the  following map:
\begin{equation}\label{ladic1}
A(F)\otimes {\mathbb Z}_l (n) \rightarrow {\Hom}_{G_{l^{\infty}}}(G_{F_{l^{\infty}}}, T_{l}(A))\otimes {\mathbb Z}_l(n),
\end{equation}
where $F_{l^{\infty}}={\bigcup_k}F_{l^k}$ and $G_{F_{l^{\infty}}}=Gal({\overline F}/F_{l^{\infty}})$.
From the map (\ref{ladic1}) by restriction of the Kummer maps to $G_{F_{l^{\infty}},S_l}$ 
we obtain the map: 
\begin{equation}
A(F)\otimes {\mathbb Z}_l (n) \rightarrow {\Hom}(G_{F_{l^{\infty}},S_l}, T_{l}(A))\otimes {\mathbb Z}_l(n)
\end{equation}

Notice that ${\Hom}(G_{F_{l^{\infty}},S_l}, T_{l}(A))\cong {H^1}(G_{F_{l^{\infty}},S_l}, T_{l}(A))$ and $H^0(G_{F_{l^{\infty}},S_l}, {\mathbb Z}_{l}(n))={\mathbb Z}_{l}(n).$
Composition with the cup product yields the map:
 \begin{equation}
  A(F)\otimes {\mathbb Z}_{l}(n) \rightarrow H^1(G_{F_{l^{\infty}},S_l}, T_l(A)(n)).
  \end{equation}
  Obviously, this map factors through $A(F_{l^{\infty}})\otimes {\mathbb Z}_{l}(n)$
 We will also need the coboundary map in the $G_{F,S_l}$-cohomology of the Kummer sequence (\ref{ks}) with $\alpha$ equal to multiplication by $l^k$:
 \begin{equation}\label{inj0}
 A(F)/l^k \hookrightarrow H^1(G_{F,S_l}, A[l^k]).
 \end{equation}
 Passing to the inverse limits in (\ref{inj0}) we obtain an injection:
 \begin{equation}\label{inj1}
 A(F)\otimes {\mathbb Z}_l \hookrightarrow H^1(G_{F,S_l}, T_l(A)).
 \end{equation} 
 \begin{rem}
  Notice that we could pass to $G_{F,S_l}$-cohomology because of the choice of $S_l$ which guarantees 
 that multiplication by $l$ on $A(F)$ is {\'e}tale (cf. Remark \ref{remm} and discussion on  p. 148  of \cite{bgk03} ).
  \end{rem}
 We have the following commutative diagram:
\begin{equation}\label{diagramk1}
\xymatrix{
A(F)\otimes {\mathbb Z}_l(n) \ar[d] \ar[r] & H^1( G_{F,S_l}, T_l(A) )\otimes Z_l(n)\ar[d]\\
A(F_{l^{\infty}})\otimes {\mathbb Z}_l(n)  \ar[r] &  \, H^1( G_{F_{l^{\infty}},S_l}, T_l(A)(n) ).
}
\end{equation} 
In the diagram (\ref{diagramk1}) horizontal arrows are injections by (\ref{inj1}) and the left vertical arrow is an obvious injection.

We also have the following:
\begin{lem}\label{keylemma}
The restriction homomorphism composed with the natural injection \linebreak $H^1( G_{F,S_l}, T_l(A)(n) ) \rightarrow H^1( G_{F_{l^{\infty}},S_l}, T_l(A)(n) )$ is a map with finite kernel.

\end{lem}
\begin{proof}
 By definitions of $S_l$ and the field $F_{l^{\infty},S_l}$ we see that 
$G_{F_{l^{\infty},S_l}}$ acts trivially on $T_l(A)(n)$ and therefore   $H^1( G_{F_{l^{\infty}},S_l}, T_l(A)(n) ) 
= Hom(G_{F_{l^{\infty}},S_l}, T_l(A)(n) ).$ We have the Leray spectral sequence in Galois cohomology (cf \cite{mi80} )
with  $E_2^{p,q}= H^p( G(F_{l^{\infty}}/F); H^q( G_{F_{l^{\infty}},S_l}, T_l(A)(n))) $ converging to $H^{p+q}(G_{F,S_l}, T_l(A)(n)).$
The inflation - restriction sequence has the following form:
\begin{equation}\label{exctseq}
 0 \rightarrow   H^1( G(F_{l^{\infty}}/F), T_l(A)(n) )  \stackrel{inf}\longrightarrow   H^1( G_{F,S_l}, T_l(A)(n) )  \stackrel{res}\longrightarrow 
 \end{equation}
 $$ H^0( G(F_{l^{\infty}}/F), {\Hom}(G_{F_{l^{\infty},S_l}}, T_l(A)(n) )).
$$
But 
$$H^0( G(F_{l^{\infty}}/F), {\Hom}(G_{F_{l^{\infty},S_l}}, T_l(A)(n) ))={\Hom}(G_{F_{l^{\infty},S_l}}, T_l(A)(n) ))^{G(F_{l^{\infty}}/F)}$$
$$\subset {\Hom}(G_{F_{l^{\infty},S_l}}, T_l(A)(n) )
= H^1( G_{F_{l^{\infty}},S_l}, T_l(A)(n) ). $$
Notice that the groups  $H^1( G(F_{l^{\infty}}/F), T_l(A)(n) )$  
    are finite ( cf.  \cite{s71}, Corollaire on p.734  ). The assertion of the lemma follows.
\end{proof}
\begin{rem}
The Corollaire  to Th{\'e}or{\'e}me 2 in \cite{s71} asserts that for any $k\ge 0$ the groups $H^k( G(F_{l^{\infty}}/F), T_l(A) )$ are finite  $l$-groups. However, the same proof shows that for any $k$ the groups $H^k( G(F_{l^{\infty}}/F), T_l(A)(n) )$ are finite $l$-groups. 
\end{rem}
\begin{cor}\label{cor2}
In particular the restriction map composed with the natural injection \linebreak $H^1( G_{F,S_l}, T_l(A) ) \rightarrow H^1( G_{F_{l^{\infty}},S_l}, T_l(A) )$ has finite kernel. Therefore if $Q_j\otimes1 \in A(F)\otimes {\mathbb Z}_l(n), j=1,\dots r$ are linearly independent elements over ${\mathbb Z}_l$  then their images  by the left verical arrow in the diagram (\ref{dgrm1}) are also linearly independent.
\end{cor}

\section{Proof of Theorem \ref{best}}
Notice that for any abelian variety we have the following commutative diagram ( cf. \cite{bgk03} diagram (6.10) ):
\begin{equation}\label{dgrm}
\xymatrix{ A(F) \ar[r]\ar[d]^{{\psi}_{F,l} }& {\prod}_{v\notin{S_l }}A_v (k_v)_l \ar[d] \\
H^1(G_{F,S_l}, T_l(A)) \ar[r] & {\prod}_v H^1(g_v,T_l(A)).
}
\end{equation}
From the diagram (\ref{dgrm}) tensoring with ${\mathbb Z}_l(n)$ we obtain the following commutative diagram:
\begin{equation}\label{dgrm1}
\xymatrix{ A(F)\otimes {\mathbb Z}_l(n) \ar[r]\ar[d]^{{\psi}_{F,l} }& {\prod}_{v\notin{S_l }}A_v (k_v)_l \otimes {\mathbb Z}_l(n)\ar[d] &\\
H^1(G_{F,S_l}, T_l(A))\otimes{\mathbb Z}_l(n) \ar[d]\ar[r] & {\prod}_v H^1(g_v,T_l(A))\otimes{\mathbb Z}_l(n) \ar[d] &\\ H^1( G_{F_{l^{\infty}},S_l}, T_l(A)(n) ) \ar[r]& {\prod}_w H^1(g_w, T_l(A_v)(n)).
}
\end{equation}
where $w|v$ are primes of ${\cal O}_{F_{l^{\infty}},S_l}.$ By Corollary \ref{cor2}  the image in $H^1(G_{F_{l^{\infty}},S_l}, T_l(A)(n))$ of any non-torsion element $Q\otimes 1 \in A(F)\otimes {\mathbb Z}_{l}(n)$ is a non-torsion element. 
 Therefore by the commutativity of the  diagram (\ref{dgrm1}) the reduction of this image in $H^1(g_w, T_l(A_v)(n))$
 comes 
from the reduction $red_v(Q\otimes 1)$. 
The proof of Theorem \ref{best} (1) is essentially the same as that of Theorem \ref{best}  (2). For (1) we use the diagram (\ref{dgrm}) whereas for the point (2) the diagram (\ref{dgrm1}). Therefore we will give a proof of Theorem \ref{best} (2). 
In what follows we  consider elements of $H^1(G_{F_{l^{\infty}},S_l}, T_l(A)(n))$ coming from $A(F)\otimes {\mathbb Z}_l(n)$.

Our proof is a generalization  of the counterexample to local - global principle for abelian varieties constructed by P. Jossen and A. Perucca in \cite{jp10} and later extended to the context of $t$-modules in \cite{bok18}.
Because of (\ref{The Zl module}) and the choice of $F$ we may assume that $J=A^{e}$  where $A$ is a geometrically simple abelian variety with ${\mathrm{End}}A={\mathbb Z}$ such that  ${d=\dim\tilde V}_{l}, e=d+1$ and ${\mathrm{rank}}A(F)\ge e.$

   Let $v\in {\cal O_{F,S_{l}}}$ be a prime of good reduction for $A.$  

 Pick non-torsion points $Q_1,\dots Q_e$ linearly independent over ${\mathbb Z}.$ Then $Q_1\otimes 1,\dots ,Q_e\otimes 1$ in $A(F)\otimes {\mathbb Z}_l(n)$  and their images $P_1,\dots P_e\in  H^1(G_{F_{l^{\infty}},S_l}, T_l(A)(n))$  are  
 linearly independent over ${\mathbb Z}_l.$   Let $\overline{P}_{i}\in H^{1}({g}_{w}; {T}_{l}(A_{v})(n ))$ be the reduction of $P_i$  ${\mod}\,\, {w.}$ 
  Consider the reductions  ${{\bar Q}_1},\dots, {{\bar Q}_e}\in A_v(k_v)_l.$


Let
\begin{align}\label{cex}
    P &:= \begin{bmatrix}
           P_{1} \\
           P_{2} \\
           \vdots \\
           P_{e}
         \end{bmatrix},
         \quad
         {\Lambda}:= \{ MP \,: \quad M\in {\mathrm{Mat}}_{e\times e}({\mathbb Z}_{l}), \quad {\tr}M=0 \}.
  \end{align}
 and   $\overline{P}=[\overline{P}_{1}\,,\dots ,\overline{P}_{e}]^{T},$  
  $\overline{P}_{i}\in H^{1}({g}_{w}; {T}_{l}(A_{v})(n ))$ be the reduction  ${\mod}\,\, {w}$  of $P.$ 
  Notice that $P\notin {\Lambda}+  H^1(G_{F_{l^{\infty}},S_l}, T_l(A)(n))_{tor}$ since $P_{1},\dots ,P_{e}$ are ${\mathbb Z}_{l}$-linearly independent.
We will find a matrix $M\in {\mathrm{Mat}}_{e\times e}({\mathbb Z}_{l})$ of trace zero such that $\overline{P}=M{\overline{P}}.$  This will show that ${\overline P}\in {\overline{\Lambda}}.$
We will give two constructions of such a matrix.

{\bf First construction}

Since the group $Y= ( {\overline P}_{1},\dots, {\overline P}_e )\subset H^{1}({g}_{w}; {T}_{l}(A_{v})(n ))$  is finite there exist ${\alpha}_{1},\dots , {\alpha}_{e}\in {\mathbb Z}$ with ${\alpha}_i$ minimal positive such that
\begin{align*} 
{\alpha}_{1}{\overline P}_{1} +   m_{1,2}{\overline P}_2 + \dots  +  m_{1,e}{\overline P}_{e}  &=  0 \\
{m_{2,1}}{\overline P}_{1} +   {\alpha}_{2}{\overline P}_2 + \dots   +  m_{2,e}{\overline P}_{e}  &=  0 \\
\dots\dots\dots\dots\dots\dots\dots\dots\dots   &    \\
{m}_{e,1}{\overline P}_{1} +   m_{e,2}{\overline P}_2 + \dots  +  {\alpha}_{e}{\overline P}_{e}  &=  0 .
\end{align*}
We will show that  $D={\gcd ({{\alpha}_1, \dots {\alpha}_e}})=1.$
Assume opposite. Choose a rational prime $p$ that divides $D.$ This means, by our choice of ${\alpha}_{1},\dots , {\alpha}_e,$  that $p$  divides  coefficients of any yielding zero linear combination of points ${\overline P}_{1},\dots, {\overline P}_e \in H^{1}({g}_{w}; {T}_{l}(A_{v})(n )).$  In particular, we see that $p$ divides the orders of ${\overline P}_i , i=1,\dots e.$ 
Since $H^{1}({g}_{w}; {T}_{l}(A_{v})(n ))$ is $l$-torsion  the  only possibility is $p=l.$


Let ${\alpha}_1=x_1l$ and
\begin{equation}\label{cexd1}
{\alpha}_{1}{\overline P}_{1}+ {x}_{2}l{\overline P}_{2}+\dots +{x}_{e}l{\overline{P_{e}}}=0
\end{equation}
be a linear relation.
 
 According to \cite{st68} (cf. also \cite{l57} ),  if $l$ is prime to the characteristic of   the field of definition $k$ of $A$  then over the separable closure $k^s$ of $k$ we have $A(k^s)[l]=({\mathbb Z}/l{\mathbb Z})^{2g}$ where $A(k^s)[l]$ denotes
 subgroup of $A(k^s)$  of points of order dividing $l$. Therefore, $A(k)[l]$ is isomorphic to one of the following groups $\{ 0, ({\mathbb Z}/l{\mathbb Z})^i,\,\,\, i=1,\dots ,2g\}$ and thus is again a ${\mathbb Z}/l{\mathbb Z}$-vector space. 
 Since $v\in {\cal O}_{F,S_l}$ we get $A(k_v)[l]\cong ({\mathbb Z}/l{\mathbb Z})^d.$
Consider  $x_1l{\bar Q}_1+\dots +{x_e}l{\bar Q}_e=0$ in $A_v(k_v)_l$

Then $$x_1{\bar Q}_1 +\dots +x_e{\bar Q}_e={\tilde T}$$ where ${\tilde T}\in \bar\Lambda\cap A(k_v)[l].$
 Now, since ${\dim}_{{\mathbb Z}/l{\mathbb Z}}A(k_v)[l]=d$ we see that ${\tilde T}$ is a linear combination of say ${\bar Q}_2,\dots, {\bar Q}_e$ and this gives the desired contradiction (cf. diagram \ref{dgrm1}).

 Hence, there exist $a_{1},\dots ,a_{e}\in {\mathbb Z}$ such that 
\begin{equation}\label{cexd2}
e=a_{1}{\alpha}_{1}+\dots  +a_{e}{\alpha}_{e}.
\end{equation}
 Put $m_{i,i}=1-a_{i}{\alpha}_{i }.$  Then $m_{1,1}+\dots +m_{e,e} =0$ and 
\begin{equation}\label{fa1}
\begin{bmatrix}
m_{1,1} & -a_1m_{1,2}& \dots & -a_1m_{1,e}\\
\dots &  \dots & \dots \\
-a_em_{e,1} & -a_em_{e,1} & \dots & m_{e,e}
\end{bmatrix}
\begin{bmatrix}
{\,{\overline P}_{1}}\\
 \dots \\
{\,{\overline P}_{e}}
\end{bmatrix}
=
\begin{bmatrix}
{\,{\overline P}_{1}}\\
 \dots \\
{\,{\overline P}_{e}}
\end{bmatrix} .
\end{equation}

Therefore ${\overline P}\in {\overline{\Lambda}}.$  
We view the matrix M in (\ref{fa1}) as the matrix with  the ${\mathbb Z}_l$ coefficients.   

{\bf Second construction}

We start with the following general group-theoretic lemma:
\begin{lem}\label{grouplemma}
Assume $G=(U_1,\dots U_{\mu})$ is an abelian $l$-torsion group generated by ${\mu}$ elements. Then any nontrivial subgroup $H=(V_1,\dots V_{\nu})$  of $G$ can be generated by at most ${\mu}$ elements. Moreover if ${\nu}>{\mu}$ then one can express
${\nu}-{\mu}$  generators say $V_{i_1}, \dots V_{i_{{\nu}-{\mu}}}$ as ${\mathbb Z}$-linear combination of the remaining generators.

\end{lem}
\begin{proof}
Assume  that  a subgroup $H$ of the group $G$ is generated by ${\nu}>{\mu}$ elements. 
Consider $G$ as a ${\mathbb Z}$-module. Then $G\cong {\mathbb Z}/{l^{n_1}}{\mathbb Z}\oplus \dots \oplus {\mathbb Z}/{l^{n_{\mu}}}{\mathbb Z}$ where $ n_1\leq n_2\leq\dots \leq n_{\mu}.$ Pick elements $S_i=(0,\dots, 0, 1, 0,\dots ,0)$  (1 on the $i$-th place)  and expand $V_i=\sum {\beta}_{j,i}^{(1)}S_j, i=1,\dots, {\nu},  j=1,\dots {\mu}$ in the ${\mathbb Z}$-basis $\{S_i\} $ of $G.$ In the columns of the matrix $({\beta}^{(1)}_{j,i})\in M_{{\mu}\times {\nu}}$ we have the coefficients of the expansion of $V_i.$ In the first row we pick a nonzero coefficient with the smallest $l$-valuation. Without loss of generality assume this element is in the first column. Subtract from each column appropriate multiple of the first column to obtain   zeros for $i>1$ in the first row. 
This is possible since  every element in the first row is of the form ${\beta}_{1,i}^{(1)}=u_i\cdot l^{k_i},$ where $u_i\in {\mathbb Z}/{l^{n_1}}{\mathbb Z}$ is a unit and $k_i <n_1.$ We obtain a new ${\mu}\times {\nu}$ matrix ${\beta}_{j,i}^{(2)}$ with one nonzero entry in the first row, namely ${\beta}^{(2)}_{1,1}={\beta}^{(1)}_{1,1}$. Now,  for $i\geq 2$ the $i$-th column contains expansion of the point $V_i- c_{i,1}^{(1)}V_1$ where $c_{i,1}^{(1)}$ is the coefficient ( viewed in ${\mathbb Z}$) by which we multiply the first column to obtain $0\in {\mathbb Z}/{l^{n_1}}{\mathbb Z}$ in the place $(1,i).$ Continue this process with the second row i.e.  pick an element ${\beta}_{j,i}^{(2)}, \, i\geq 2$ with the smallest $l$-valuation, transpose columns and enumerate points correspondingly so that the element with the smallest $l$-valuation is at the place $(2,2).$ Obtain zeros for columns with the index $i>2$. Continue the process until we obtain the lower diagonal matrix. Notice that at the $k$-th stage  for $i\geq k$ the  $i$-th column contains the expansion of the element 
 $V^{(k)}_{i} -{\sum}_{s<i}c_{i,s}^ {(k)}V^{(k)}_{s}$ in the basis $\{S_j, \, j=1,\dots ,{\mu}\}.$  
 Assume that at the $m$-th stage we obtain zeros in the columns $i> r$ where $r\leq {\mu}.$
 Let $C=\{V^{(m)}_{r+1}, \dots , V^{(m)}_{{\nu}}\}$  be the elements corresponding to the zero columns.
 Thus any element from $C$ is a linear combination of the generators not belonging to $C.$
 Since our possible transposition at the $k$-th stage  involved transposition of the $k$-th and $i$-th columns with $i>k$ we see that the coefficient at  each $V^{(m)}_{r+1},\dots ,V^{(m)}_{{\nu}} $ is one. 
 Thus for $i\in C$ we have  $V^{(m)}_i={\sum}_{s\notin  C} c_{i,s}^{(m)}V^{(m)}_{s}$ 
 \end{proof}
Now  we can apply Lemma \ref{grouplemma} to  $G=\Im(A_v(k_v)_l\otimes {\mathbb Z}_l(n))\subset H^{1}({g}_{w}; {T}_{l}(A_{v})(n ))$ (cf. diagram (\ref{dgrm1})) and $H=(
{\bar P}_1,\dots , {\bar P}_e).$ Since $G$ is a finite $l$-group generated by $d$ elements and $e=d+1$ we may assume without loss of generality that ${\bar P}_1={\sum}_{i=2}^e c_i{\bar P}_i.$ Then we can choose
\[M=\left[ \begin{array}{ccccc}
	1-e & e c_{2} & e c_{3} & \ldots\ldots & e c_{e}\\
	0 &1 & 0 & \ldots\ldots  & 0\\
		0 &0 & 1 & \ldots\ldots  & 0\\
		\vdots &\vdots & \vdots &  \ddots & \vdots\\
			0 &0 & 0 & \ldots\ldots  & 1\\
\end{array}\right]. \] 

Obviously ${\overline P}=M{\overline P}$ and therefore ${\overline P}\in {\overline{\Lambda}}$ but $P\notin {\Lambda}+  H^1(G_{F_{l^{\infty}},S_l}, T_l(A)(n))_{tor}.$

\section{Local to global principle and dynamical systems}
S. Bara{\'n}czuk  in  \cite{b17} considers abelian groups satisfying the following two axioms. 

\begin{ass}\label{ass1}
Let $B$ be an abelian group such that there are homomorphisms $r_{v}: B\rightarrow B_{v}$ for an infinite family $v$, whose targets $B_v$ are finite abelian groups.
\begin{enumerate}
\item[(1)] Let $l$ be a prime number, $(k_{1},\dots , k_m)$  a sequence of nonnegative integers. If $P_1,\dots , P_m \in B$ are points linearly independent over ${\mathbb Z},$ then there is a family of primes $v$ in $F$ such that 
$l^{k_{i}}|| {\ord}_{v}P_{i}$ if $k_{i}>0$ and $l\nmid {\ord}_{v}P_i$ if $k_i=0.$
\item[(2)] For almost all $v$ the map $B_{\text{\rm tors}}\rightarrow B_{v}$ is injective.
\end{enumerate}
\end{ass}
Here ${\ord}_{v}P$ is the order of a reduced point $P\mod v$.
Under the  Assumptions \ref{ass1}, S. Bara{\'n}czuk was able to prove the following dynamical version  of the local lo global principle
\begin{thm}\label{b17}( \cite{b17})
Let $\Lambda$ be a subgroup of $B$ and  $P\in B$ be a point of infinite order and $\phi$ be a natural number. Then the following are equivalent:
\begin{enumerate}
\item[i)]
For almost every $v$
$$O_{\phi}(P \mod v) \cap ({\Lambda} \mod v)\neq\emptyset ,$$
\item[ii)]
$O_{\phi}(P) \cap ({\Lambda})\neq\emptyset$.
\end{enumerate}
\end{thm}
Here $O_{\phi}(P)=\{ {\phi}^n(P) : n\geq 0\}.$
We have the following lemma:
\begin{lem}\label{kg}
Let $A=A_{1}^{e_1}\times \dots \times A_{t}^{e_t},$ be an abelian variety such that ${\End}A_i={\mathbb Z} $ for $i=1,\dots ,t.$
 Then $B=H^{1}(G_{F, S_{l}}; {T}_{l}(A) (n))$ and $B_v=H^{1}({g}_{v}; {T}_{l}(A_{v})(n ))$ for $v\in {\cal O}_v$ fulfil Assumptions \ref{ass1}.
\end{lem}
\begin{proof}
Assumption (1) of \ref{ass1} is a specialization to ${\cal R}_i={\mathbb Z}$ of Corollary 3.5  of {\cite{bk13}}. 
Assumption (2) is an assertion of Lemma \ref{inj}.
\end{proof}

\begin{thm}\label{theorem 5.1}
Let $X/F$ be a smooth, proper and geometrically irreducible curve of genus $g.$ Let $F^{\prime}$ be an isogeny splitting field of the Jacobian $J$ and 
assume that for the corresponding product $A_{1}^{e_{1}}\times \dots \times A_{t}^{e_{t}},$ we have $End_{{\overline F}} \, A_{i}={\mathbb Z}$ for each $1 \leq i \leq t.$ Let $l > 2$ be a prime number which is coprime to the polarisation degrees of the abelian varieties 
$A_i.$ Let $S_{l}$ be a set of places of $F$ containing the places of bad reduction, the  
archimedean places and the primes above $l.$ Let  ${\cal X}$ be a regular and proper model of $X$ over ${\cal O}_{F, S_l}.$ 
 Let $\hat{P} \in K_{2n}^{et}({\cal X})$ be a point of infinite order and let $\hat{\Lambda}$ be a finitely generated 
 ${\mathbb Z}_{l}$-submodule of $ K_{2n}^{et}({\cal X}).$
   Let $w$ be a natural number and $O_{w}= \{ w^{k}P: k\geq 0 \}$ Then the following
 are equivalent
\begin{enumerate}
\item[(1)] For almost every prime ${v} \in  {\cal O}_{F,S_l}$
$$O_{w}({\mathrm{r}}_{v}(P))\cap {\mathrm{r}}_{v}({\Lambda}) \neq \emptyset ,$$
\item[(2)] $O_{w}(P)\cap {\Lambda}^{\prime}\neq\emptyset$ where ${\Lambda}^{\prime}$ is a subgroup of 
$ K_{2n}^{et}({\cal X})$ generated by $\Lambda$ and the finite kernel of the map 
$K_{2n}^{et}({\cal X}) \rightarrow
H^{1}(G_{S_{l}}; {T}_{l}(J)(n)).$

\end{enumerate}
 
\end{thm}
\begin{proof}
The proof follows from Lemma \ref{kg} , commutativity of the diagram (\ref{diagram 2.4}). \end{proof}
\begin{rem}\label{finrem}
If the Bass Conjecture and Quillen-Lichtenbaum Conjecture hold true for ${\cal X}$ then we obtain ( using diagram (\ref{diagram 2.4a}) instead of (\ref{diagram 2.4}) \,) the corresponding to Theorem \ref{theorem 5.1}
statement for Quillen $K$-theory of ${\cal X}.$
\end{rem}


{\bf Acknowledgments}
\rm 
We  would   like to thank G. Banaszak for helpful discussions and D. Ulmer for drawing our attention to 
\cite{hn65} and \cite{ou73}. We would also  like to thank the referee whose remarks helped to improve the manuscript.

{}

\end{document}